\documentclass[12pt]{article}
\usepackage{amsmath,amssymb,amsthm,graphicx}
\usepackage{enumitem}

\usepackage{hyperref}
\usepackage{fullpage}
\usepackage{caption}
\usepackage{subcaption}
\usepackage{float}
\usepackage{amsmath}                            
\usepackage{amssymb}                            
\usepackage{amsthm}                             
\usepackage{graphicx}   
\usepackage{todonotes}                        
\usepackage{amsfonts}
\usepackage{color}
\usepackage{xcolor}
\usepackage{xspace}
\usepackage[font=small, margin=15pt, labelfont=bf]{caption}
\usepackage{microtype}
\usepackage[mathscr]{eucal}
\usepackage{float}
\usepackage{amsthm}

\DeclareMathOperator{\dcdot}{\cdot\cdot}
\DeclareMathOperator{\Seq}{\subseteq}
\newtheorem{theorem}{Theorem}[section]
\newtheorem{lemma}[theorem]{Lemma}
\newtheorem{corollary}[theorem]{Corollary}

\newtheorem{prop}[theorem]{Proposition}

\theoremstyle{theorem}
\newtheorem{thmx}{Theorem}

\newcommand{\brac}[1]{\left\lbrace #1 \right\rbrace}
\newcommand{\card}[1]{\left| #1 \right|}

\newcommand{\Case}[2]{\noindent {\bf Case #1:} \emph{#2}}
\newcommand{\Subcase}[2]{\noindent {\bf Subcase #1:} \emph{#2}}

\theoremstyle{definition}
\newtheorem{defn}{Definition}[section]

\theoremstyle{remark}
\newtheorem*{remark}{Remark}

\theoremstyle{definition}

\newtheorem*{notation}{Notation}
\theoremstyle{theorem}

\newcounter{claimCount}
\newenvironment{claim}{\medskip

    \noindent\refstepcounter{claimCount}\textbf{Claim~\arabic{claimCount}.}}{

    \medskip}
\newenvironment{claimproof}{\noindent\textit{Proof of Claim~\arabic{claimCount}.}}{\hfill\ensuremath{\qedsymbol} \normalsize{Claim~\arabic{claimCount}}

    \medskip}

\begin{document}

\title{On the structure of (claw, bull)-free graphs}

\author{\large{Sebasti\'an Gonz\'alez Hermosillo de la Maza}\\
Department of Mathematics\\
Simon Fraser University\\
Burnaby, BC, Canada\\
{\texttt{sga89@sfu.ca}}
\and
\large{Yifan Jing}\\
Department of Mathematics\\
University of Illinois at Urbana Champaign\\
Urbana, IL, USA\\
{\texttt{yifanjing17@gmail.com}}
\and
\large{Masood Masjoody}\\
Department of Mathematics\\
Simon Fraser University\\
Burnaby, BC, Canada\\
{\texttt{mmasjood@sfu.ca}}
}

\date{}

\maketitle

\begin{abstract}
In this research, we determine the structure of (claw, bull)-free graphs. We show that every connected (claw, bull)-free graph is either an expansion of a path, an expansion of a cycle, or the complement of a triangle-free graph; where an expansion of a graph $G$ is obtained by replacing its vertices with disjoint cliques and adding all edges between cliques corresponding to adjacent vertices of $G$. This result also reveals facts about the structure of triangle-free graphs, which might be of independent interest.


\end{abstract}

\section{Introduction}

The structure of graphs with some given forbidden subgraphs is well studied, and quickly gained several applications in graph theory and in theoretical computer science. For some of the known results in this field see \cite{brandstadt1999graph}, and \cite{isgci}.

In this paper, we study the structure of (claw, bull)-free graphs. A graph is a \emph{claw} if it is isomorphic to $K_{1,3}$, and a \emph{bull} if it can be obtained from a triangle by adding two pendant edges at two different vertices (Figure \ref{fig}).

\begin{figure}[h]
\centering
\begin{minipage}[t]{0.45\textwidth}
\centering
\includegraphics[width=1.5in]{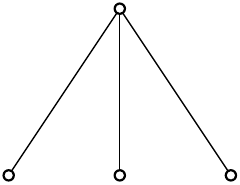}
\caption*{Claw}
\end{minipage}\hfill\begin{minipage}[t]{0.55\textwidth}
\centering
\includegraphics[width=1.5in]{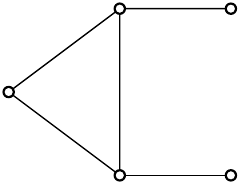}
\caption*{Bull}
\end{minipage}
\caption{Claw and bull}\label{fig}
\end{figure}

\begin{defn}\label{defn: Expansion}
An {\em expansion} of a graph $G$ with vertex set $V(G)=\{v_1,\dots,v_n\}$ is any graph $H$ obtained from $G$ by substituting its vertices with disjoint cliques $K^{[i]}$, $i=1,\dots,n$, (called the {\em bags} of the expansion) and adding the edges of the complete bipartite graphs with the partite sets $V(K^{[i]})$ and $V(K^{[j]})$ for each $v_i v_j\in E(G)$.
\end{defn}
\begin{notation}
Let $X$ and $Y$ be disjoint subsets of the vertex set of a graph $G$. Then we write $X\Leftrightarrow_G Y$ (or simply $X\Leftrightarrow Y$ if the graph $G$ is understood from the context) to mean that every vertex in $X$ is adjacent to every vertex in $Y$. We also denote by $N_G(X)$ (or $N(X)$, if $G$ is understood) the open neighborhood of $X$, defined by $$N_G(X)=\Big(\bigcup_{x\in X} N_G(x)\Big)\setminus X.$$ Furthermore, given a natural number $n\in\mathbb{N}$ we write $[n]:=\{1,2,\dots,n\}$.
\end{notation}

\noindent The main result in this paper is as follows:
\begin{theorem}\label{thm:main}
	A connected graph $G$ is (claw, bull)-free graph if and only if it belongs to one of the following (disjoint) classes of graphs:
	
	\begin{itemize}
		\item the class of graphs which are expansions of paths of length at lesat four,
		\item the class of graphs which are expansions of cycles of length at least six,
		\item the class of connected graphs which are complements of triangle-free graphs.
		
			\end{itemize}
\end{theorem}

Since the complement of a bull is still a bull, the complement of triangle free graphs are also (claw, bull)-free. As a corollary of Theorem \ref{thm:main}, this two classes are almost the same.

\begin{corollary}
The class of triangle-free graphs is the union of the class of complete bipartite graphs and the class of complements of all graphs $G$ where $G$ is a connected (claw,bull)-graph which is not an expansion of a path of length at least $4$ or an expansion of a cycle of length at least $6$.
\end{corollary}

In the following three sections we consider three sub-classes of (claw, bull)-free graphs based on the length of a longest cycle. Finally we combining the result of these sections to show Theorem \ref{thm:main}. We will use standard definitions and notation for graphs as given in \cite{bondy2008graph}.

\begin{notation}
Given a graph $G$, we define $\ell(G)$ as the length of a longest induced cycle in $G.$
\end{notation} 

This research was inspired by the following result on the structure of (claw,bull)-free graphs, obtained in a study of the game of cops and robbers\cite{Aigner}:

\begin{lemma}\cite{Masood3}\label{lemma: structure cw-bl-free}
Let $u_0$ and $u_1$ be two adjacent vertices in a (claw,bull)-free graph $G$, and  let $U$ be the set of neighbor of $u_0$ in $G-u_1$. Then, the component $H$  of $u_0$ in $G-U$ is an expansion of a path whose bags; in other words, with $N_0=\{u_0\}$ and $N_i$ being  the $i$th neighborhood of $u_0$ in $H$ for each positive integer $i$, each $N_i$ is a clique and we have $N_i\Leftrightarrow N_{i-1}$ for each $i\ge 1$.
\end{lemma}

Indeed, we shall use Lemma \ref{lemma: structure cw-bl-free} to show that the sub-class of (claw,bull)-free graphs under consideration in Section \ref{section: path-expansions} consists of expansions of paths. 

\section{The case $\ell(G) \geq 6$.}

\begin{lemma}\label{neighboursincycles}
	

Let $G$ be a (claw, bull)-free graph, $C$ an induced cycle of length $k\ge4$ and $x \in N(C)$. Then $N(x)$ contains two consecutive vertices of $C$. Moreover, if $k\ge 5$ then $N(x)$ contains three consecutive vertices of $C$.

\end{lemma}

\begin{figure}[h]
\begin{center}
 \includegraphics[scale=1.25]{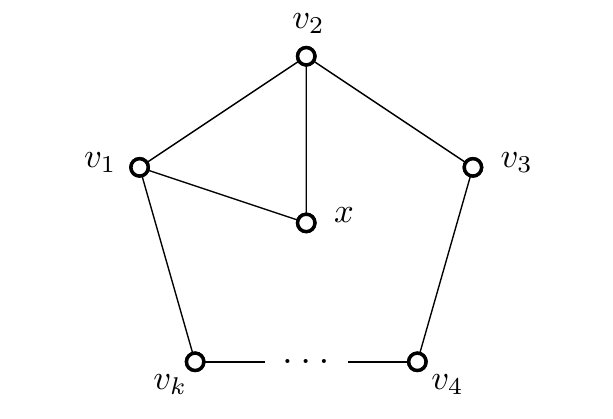}
 \caption{Consecutive neighbors for vertices in $N(C)$ where $C$ has length $\ge 4$}\label{pic: c-b-free-st01}
\end{center}
\end{figure}

\begin{proof}
	

Let $V(C)=\{v_1,\dots,v_k\}$ and suppose $xv_1\in V(G)$. Since $G$ is claw-free, we must have $xv_2\in E(G)$ or $xv_k\in E(G)$, establishing the first claim. Suppose, without loss of generality, that $xv_2\in E(G)$. Then, in case $k\ge 5$ one must have $xv_3\in E(G)$ or $xv_k\in E(G)$, for otherwise $G[\{x,v_1,v_2,v_3,v_k\}]$ would be a bull (See Figure \ref{pic: c-b-free-st01}).
\end{proof}

\begin{lemma}\label{dominatingcycle}
Let $G$ be a connected (claw, bull)-free graph, and $C$ an induced cycle of  length $\ge 4$. Then $N[C]=V(G)$.
\end{lemma}

\begin{figure}[h]
\centering
\begin{minipage}[t]{0.50\textwidth}
\centering
\includegraphics[width=2.9in]{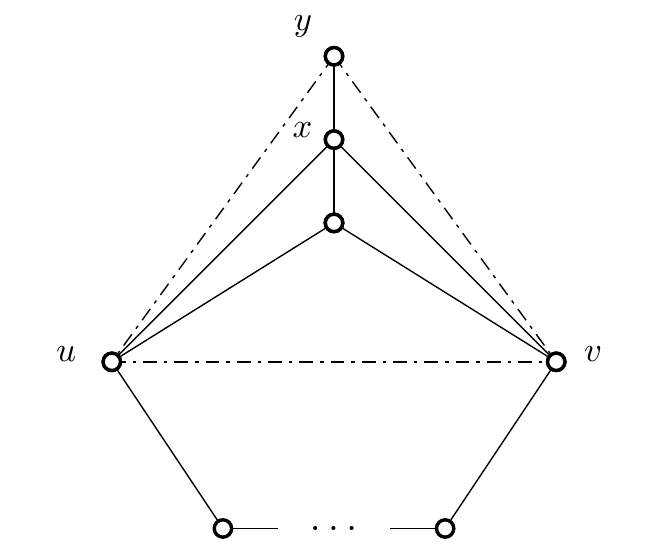}
\caption{Proof of Lemma \ref{dominatingcycle}; the case $\card{N(x)\cap V(C)}\ge 3.$}
\end{minipage}\hfill\begin{minipage}[t]{0.50\textwidth}
\centering
\includegraphics[width=2.9in]{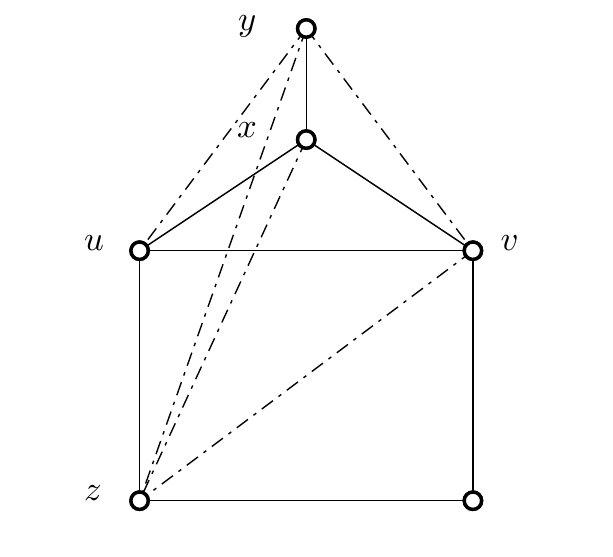}
\caption{Proof of Lemma \ref{dominatingcycle}; the case $\card{N(x)\cap V(C)}= 2.$}
\end{minipage}
\end{figure}

\begin{proof}
	
Suppose $N[C]\not=V(G)$. Choose a vertex $y$ at distance two from $C$ and a vertex $x\in N(C)\cap N(y)$. If $N(x) \cap V(C) \geq 3$, then there exist two vertices $u,v \in N(x) \cap V(C)$ which are not adjacent, in which case $\brac{x,y,u,v}$ induces a claw, a contradiction. Therefore, according to Lemma \ref{neighboursincycles}, $C$ is a cycle of length $4$ such that $N(x)\cap V(C)$ consists of two consecutive vertices of $C$. Then $G[\{N(x)\cap V(C)\}\cup\{x,y,z\}]$ is a bull for $z\in V(C)\setminus N(x)$; a contradiction. 
\end{proof}

\begin{lemma}\label{longcycles}
	
Let $G$ be a connected (claw, bull)-free graph and $C$ an induced cycle of $G$ of length $k$. If $k \geq 6$, then $G$ is an expansion of $C$.
	
\end{lemma}

\begin{proof}
\setcounter{claimCount}{0}
	Let $V(C) = \brac{v_1, v_2, \dots, v_k}$. By Lemma \ref{dominatingcycle} we know that $N[C] = V(G)$ and every vertex outside of $C$ has at least three neighbours in $C$. 
\begin{claim}\label{claim: 0} Let $x\in V(G)\setminus V(C)$. Then $\card{N(x)\cap V(C)}=3$.
\end{claim}
\begin{claimproof}
If $\card{N(x)\cap V(C)}=5$, then $N(x)$ would contain an independent set of size three, i.e. $G[N[x]]$ would contain a claw. Hence, proceeding by the way of contradiction and in light of Lemma \ref{neighboursincycles} we may assume $\card{N(x)\cap V(C)}=4$. As such, without loss of generality we may assume $N(x)=\{v_1,v_{a},v_{b},v_{c}\}$ where $1<a<b<c<k$. Note that $1,a,b,c$ cannot be consecutive for otherwise $G[\{x,v_1,v_2,v_4,v_k\}]$ would be bull. Moreover, if $a>2$ (resp. $c>b+1$) then $G[\{x,v_1,v_a,v_c\}]$ (resp. $G[\{x,v_1,v_b,v_c\}]$) would be a claw, a contradiction. Hence, one must have $a=2$, $b>3$ and $c=b+1$. But then $G[\{v_1,v_2,v_b,v_k,x\}]$ would be a bull, a contradiction.
\begin{figure}[h]
\centering
\begin{minipage}[t]{0.50\textwidth}
\centering
\includegraphics[width=2.9in]{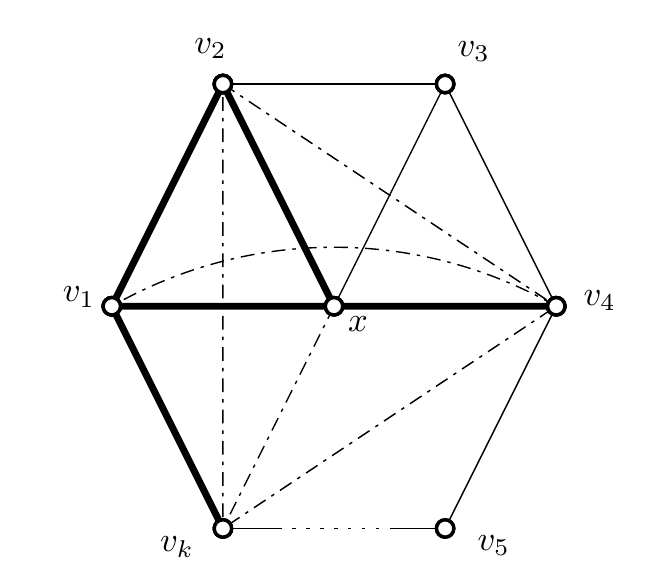}
\caption*{\textbf{(a)} The case $a=2$, $b=3$, $c=4$}
\end{minipage}\hfill\begin{minipage}[t]{0.50\textwidth}
\centering
\includegraphics[width=2.9in]{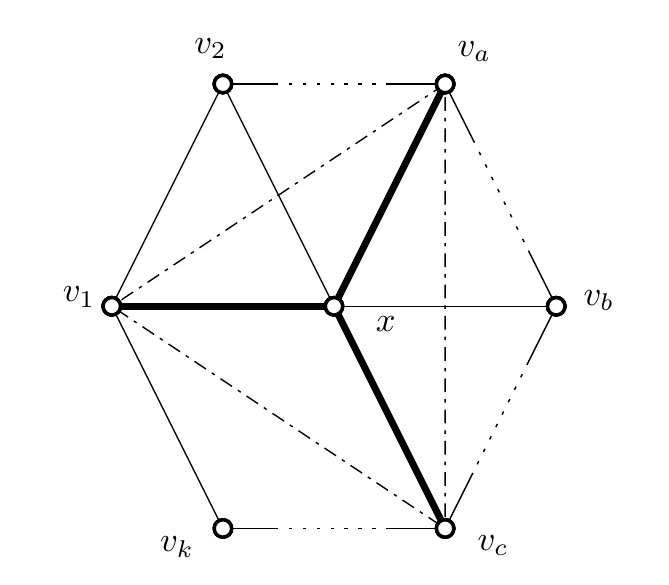}
\caption*{\textbf{(b)} The case $a>2$ (the case $c>b+1$ is similar)}
\end{minipage}
\caption{Ruling out the case $\card{N(x)\cap V(C)}=4$ in Claim \ref{claim: 0}, Lemma \ref{longcycles}, by considering $N(x)=\{v_1,v_{a},v_{b},v_{c}\}$ where $1<a<b<c<k$}\label{pic: Fig-c-b-free-st03-cl1}
\end{figure}
\end{claimproof}
For the rest of the proof, set $N_x:=N(x)\cap V(C)$ for each $x\in V(G)$.

\begin{claim}\label{claim: 1} Let $x,y$ be distinct vertices of $G$ such that $\card{N_x \cap N_y} \geq 2$. Then $xy\in E(G)$.
\end{claim}
\begin{figure}[h]
\centering
\begin{minipage}[t]{0.50\textwidth}
\centering
\includegraphics[width=2.9in]{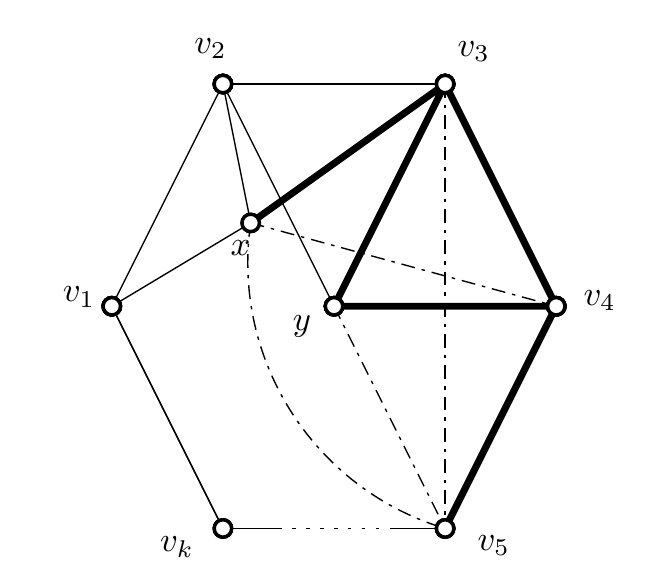}
\caption*{\textbf{(a)} The case $\card{N_x\cap N_y}=2$. With $N_x=\{v_1,v_2,v_3\}$ and $N_y=\{v_2,v_3,v_4\}$, $G[\{x,y,v_3,v_4,v_5\}]$ would be a bull unless $xy\in E(G)$.}
\end{minipage}\hfill\begin{minipage}[t]{0.50\textwidth}
\centering
\includegraphics[width=2.9in]{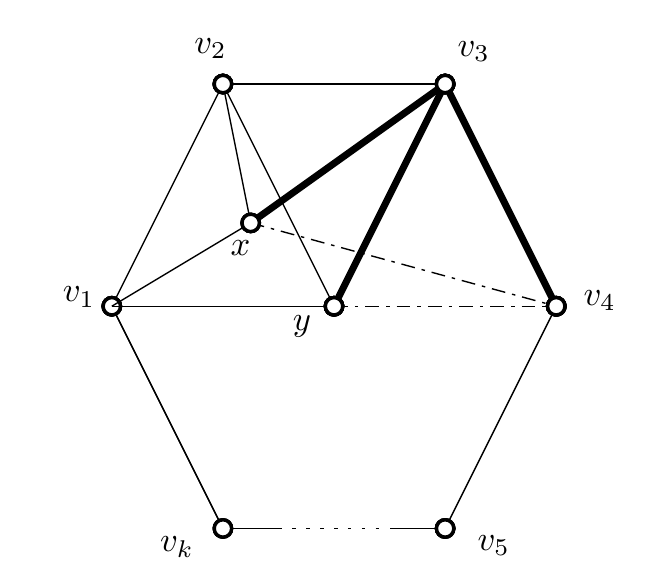}
\caption*{\textbf{(b)}  The case $\card{N_x\cap N_y}=2$. With $N_x=N_y=\{v_1,v_2,v_3\}$, $G[\{x,y,v_3,v_4\}]$ would be a claw unless $xy\in E(G)$.}
\end{minipage}
\caption{Cases I and II in Claim \ref{claim: 1}, Lemma \ref{longcycles}}\label{pic: Fig-c-b-free-st04-cl2}
\end{figure}
\begin{claimproof}
If $x\in V(C)$ then $y\in V(G)\setminus V(C)$; thereby, according to Lemma \ref{neighboursincycles}, $N_y$ consists of three consecutive vertices of $C$. Hence by Claim \ref{claim: 0} we have $N_y=N_x\cup \{x\}$ and; in particular, $xy\in E(G)$. Hence, we may assume $x,y\in V(G)\setminus V(C)$. Suppose, contrary to the claim, that $xy\notin E(G)$.
\medskip

\Case{I}{$\card{N_x\cap N_y}=2$.}

Let $N_x=\{v_1,v_2,v_3\}$ and $N_y=\{v_2,v_3,v_4\}$. As such, $G[\{x,y,v_3,v_4,v_5\}]$ would be a bull unless $xy\in E(G)$.
\medskip

\Case{II}{$\card{N_x\cap N_y}=3$.}

In this case, $N_x$ and $N_y$ are the same set, say, $\brac{v_1,v_2,v_3}$. As such, $G[\{x,y,v_3,v_4\}]$ would be a claw unless $xy\in E(G)$.
\end{claimproof}
For each $i\in[1\dcdot k]$ set $C_i=\{x\in V(G): N_x\supseteq N_{v_i}\}$. According to Lemma \ref{neighboursincycles} $C_i$s partition $V(G)$. Furthermore, in light of Claim \ref{claim: 1} it follows that:
\begin{itemize}
\item each $C_i$ is a clique,
\item $E[C_i,C_j]$ is a complete bipartite graph if $v_i$ and $v_j$ are consecutive vertices of $C$, and
\item  $E[C_i,C_j]$ has no edge if $v_i$ and $v_j$ are distinct nonconsecutive vertices of $C$;
\end{itemize}
from which it follows that $G$ is an expansion of $C$, as desired.	
\end{proof}

\section{The case $\ell(G) \in\{4,5\}$.}

\begin{lemma}\label{lemma: comp. tri.free}
	
	Let $G$ be a (claw, bull)-free graph. If $\ell(G)=4$ or $5$, then the maximum size of an independent set of vertices in $G$ is at most $2$. 
	
\end{lemma}

\begin{proof}
Let $I$ be a largest independent set in $G$ with $\card{I}\ge 3$.
\medskip

\Case{1}{$\ell(G)=4$.}

Let $C=v_1,v_2,v_3,v_4,v_1$ be an induced cycle in $G$. Since $\alpha(C)=2$, $\card{I\cap V(C)}\in\{0,1,2\}$.
\smallskip

\Subcase{1.1}{$\card{I\cap V(C)}=2$.}

According to Lemmas \ref{neighboursincycles} and \ref{dominatingcycle} every vertex $x\in I\setminus V(C)$ is adjacent to two consecutive vertices of $C$. Hence, $I\cap V(C)$ has to consist of two consecutive vertices of $C$, a contradiction.
\smallskip

\Subcase{1.2}{$\card{I\cap V(C)}=1$.}

Let $I\cap V(C)=v_1$ and $x,y$ be distinct vertices in $I\setminus V(C)$. Without loss of generality, suppose $v_2,v_3\in N(x)$. Note that if $v_2y\in E(G)$, then $G[\{v_1,v_2,x,y\}]$ would be a claw. Hence, we must have $v_3,v_4\in N(y)$. But then $G[\{v_1,v_2,v_3,x,y\}]$ would be a bull, a contradiction.
\smallskip

\begin{figure}[h]
\begin{center}
 \includegraphics[scale=1]{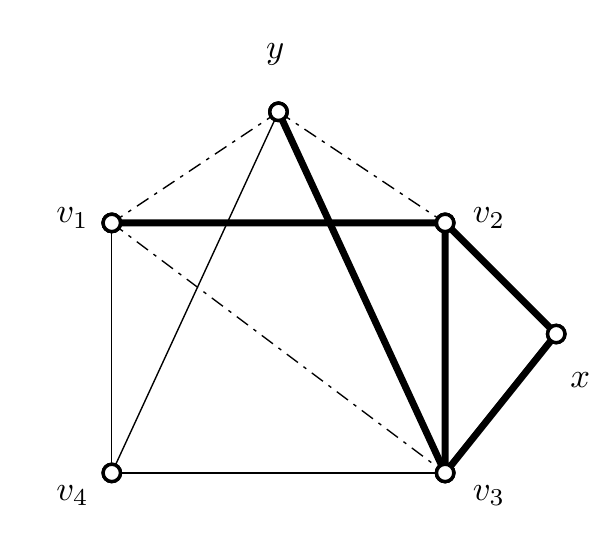}
 \caption{Subcase 1.2, Lemma \ref{lemma: comp. tri.free}. With $\{x,y,v_1\}\Seq I$ one must have $v_2y\not\in E(G)$; thereby, $\{v_3,v_4\}\Seq N(y)$. Then $G[\{v_1,v_2,v_3,x,y\}]$ will be a bull.}\label{pic: c-b-free-st01-1.2}
\end{center}
\end{figure}

\Subcase{1.3}{$\card{I\cap V(C)}=0$.}

Let $x,y,z$ be distinct vertices in $I$. Since $G$ is claw-free, no vertex of $C$ is adjacent to all three of $x,y,z$. Hence, by the pigeonhole principle and Lemmas \ref{neighboursincycles} and \ref{dominatingcycle}, we may assume $v_3\notin N(x)$ and $v_4\notin N(x)$, which imply $xv_1,xv_2\in E(G)$. Furthermore, we may assume $v_1\notin N(y)$ (See Figure \ref{pic: c-b-free-st01-1.3}). If in addition $v_4\notin N(y)$, we would have $v_2y,v_3y\in E(G)$ in which case $G[\{v_2,v_3,v_4,x,y\}]$ would be a bull. Hence, $v_4y\in E(G)$, which in turns implies $v_3y\in E(G)$ (according to Lemma \ref{neighboursincycles}). Now observe that if $v_1,v_4\in N(z)$ then $G[\{v_1,v_4,x,y,z\}]$ would be a bull, and if only one of $v_1,v_4$ is in $N(z)$ then $G[\{v_1,v_4,x,z\}]$ or $G[\{v_1,v_4,y,z\}]$ would be a claw. Hence, $v_1\notin N(z)$ and $v_4\notin N(z)$; thereby $v_2,v_3\in N(z)$. But then $G[\{v_2,v_3,v_4,x,z\}]$ would be a bull, a contradiction. 
\medskip

\begin{figure}[h]
\begin{center}
 \includegraphics[scale=1]{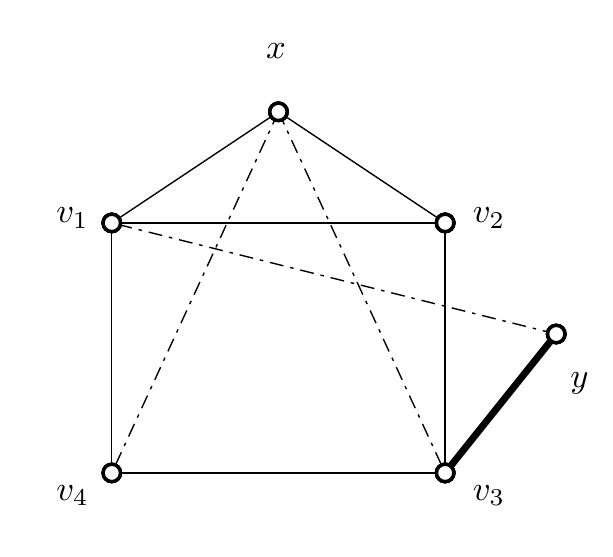}
 \caption{General situation in Subcase 1.3, Lemma \ref{lemma: comp. tri.free}. With $\{x,y,z\}\Seq I$, one may assume $xv_3\not\in E(G)$, $xv_4\not\in E(G)$, implying $xv_1,xv_2\in E(G)$. One may further assume $yv_1\not\in E(G)$. As such, Lemmas \ref{neighboursincycles} and \ref{dominatingcycle} imply $yv_3\in E(G)$. }\label{pic: c-b-free-st01-1.3}
\end{center}
\end{figure}

\begin{figure}[hp]
\centering
\begin{minipage}[t]{0.50\textwidth}
\centering
\includegraphics[width=2.6in]{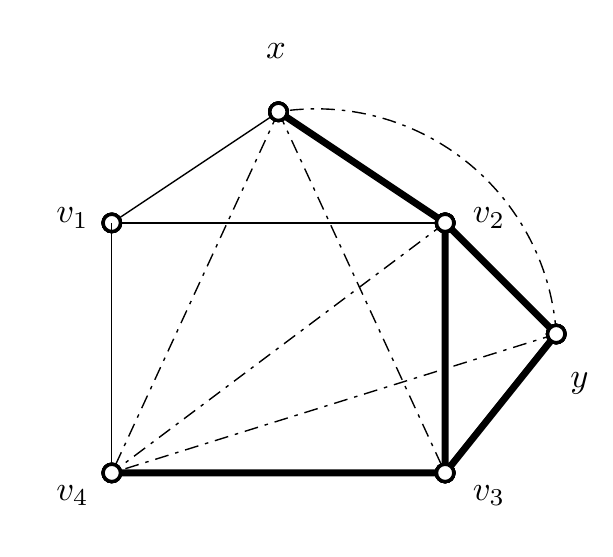}
\caption*{\textbf{(a)} If $v_4y\not\in E(G)$, then $yv_2\in E(G)$; thereby, $G[\{v_2,v_3,v_4,x,y\}]$ would be a bull.}
\end{minipage}\hfill\begin{minipage}[t]{0.50\textwidth}
\centering
\includegraphics[width=2.6in]{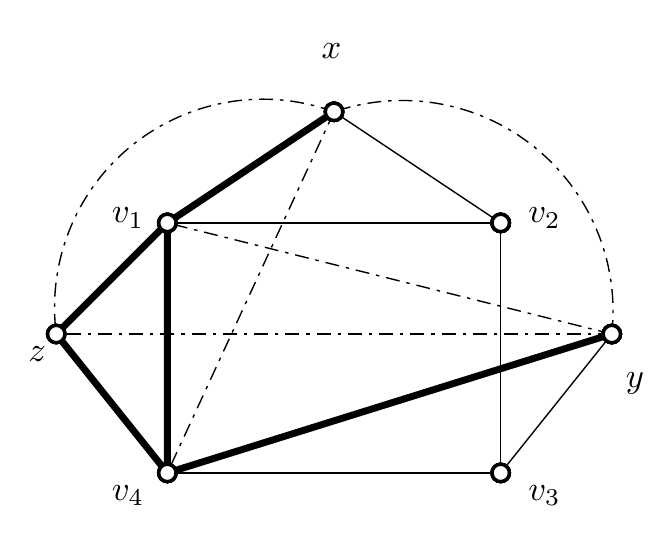}
\caption*{\textbf{(b)} If $v_1z,v_4z\in E(G)$, $G[\{v_1,v_4,x,y,z\}]$ would be a bull.}
\end{minipage}\hfill\begin{minipage}[t]{0.50\textwidth}
\centering
\includegraphics[width=2.6in]{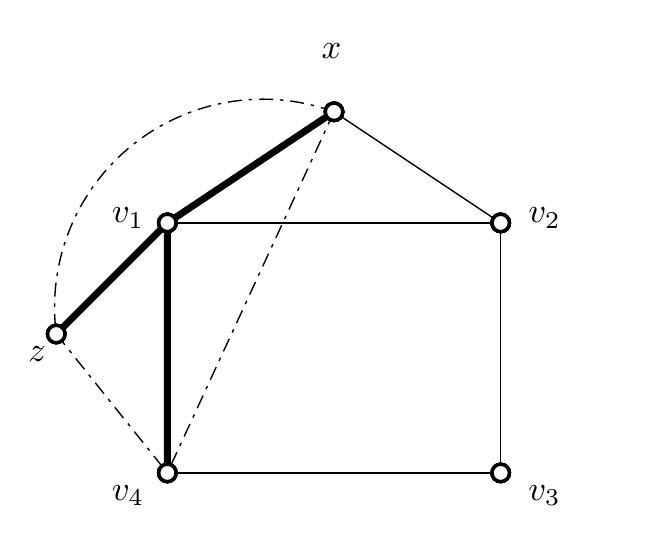}
\caption*{\textbf{(c)} If $v_1z\in E(G)$ and $v_4z\not\in E(G)$ then $G[\{v_1,v_4,x,z\}]$ would be a claw.}
\end{minipage}\hfill\begin{minipage}[t]{0.50\textwidth}
\centering
\includegraphics[width=2.6in]{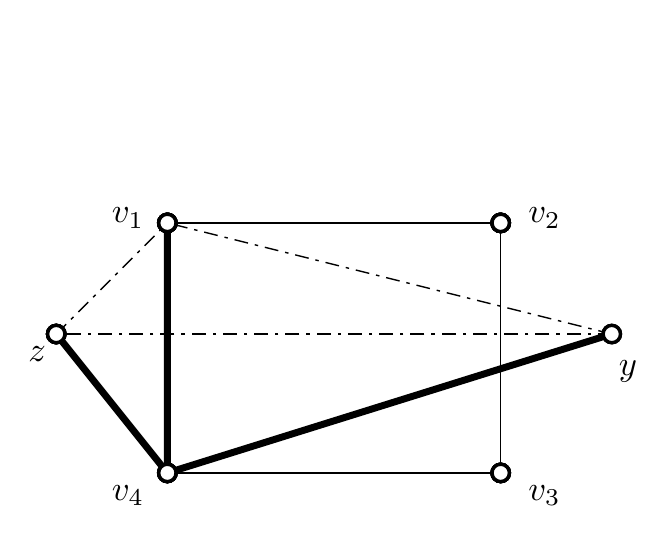}
\caption*{\textbf{(d)} If $v_4z\in E(G)$ and $v_1z\not\in E(G)$ then $G[\{v_1,v_4,y,z\}]$ would be a claw.}
\end{minipage}\hfill\begin{minipage}[t]{0.50\textwidth}
\centering
\includegraphics[width=2.6in]{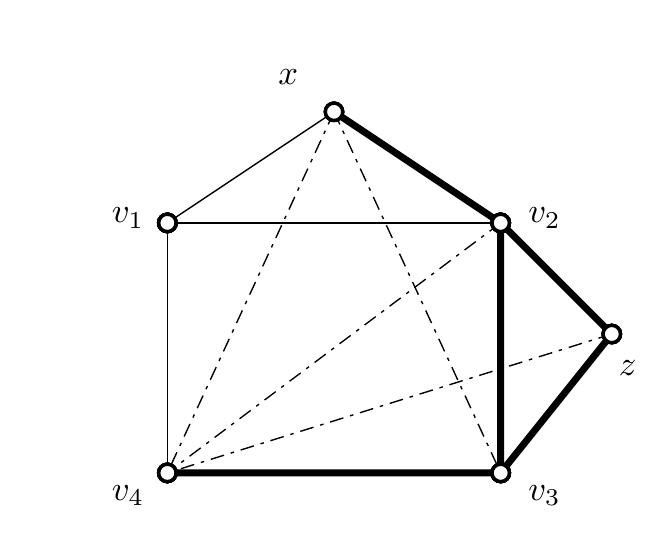}
\caption*{\textbf{(e)} Since $v_2z,v_3z\in E(G)$ and $v_4z\not\in E(G)$, $G[\{v_2,v_3,v_4,x,z\}]$ is a bull, a contradiction.}
\end{minipage}
\caption{Lemma \ref{lemma: comp. tri.free} with $\ell(G)=4$. From the general situation described in Figure \ref{pic: c-b-free-st01-1.3} one gets $\{v_3,v_4\}\Seq N(y)$ and $N(z)\cap V(C)=\{v_2,v_3\}$, leading to the bull in \textbf{(e)}}\label{pic: Fig-c-b-free-st01-l4}
\end{figure}

	
	
	
	
	
	

\Case{2}{$\ell(G)=5$}

Let $C=v_1,v_2,v_3,v_4,v_5,v_1$ be an induced cycle in $G$.
\smallskip

\Subcase{2.1}{$\card{I\cap V(C)}=2$.}

Every vertex $x\in I\setminus V(C)$ is adjacent to three consecutive vertices of $C$. Hence, likewise Case 1.1, $I$ has to contain two consecutive vertices of $C$, a contradiction.
\smallskip

\Subcase{2.2}{$\card{I\cap V(C)}=1$.}

Let $I\cap V(C)=\{v_1\} $ and $x,y$ be distinct vertices in $I\setminus V(C)$. Without loss of generality, suppose $v_2,v_3,v_4\in N(x)$. If $v_2y\in E(G)$, then $G[\{v_1,v_2,x,y\}]$ would be a claw. Hence, $N(y)\cap V(C)=\{v_3,v_4,v_5\}$. But then $G[\{v_1,v_2,v_3,x,y\}]$ would be a bull, a contradiction.
\smallskip

\begin{figure}[H]
\begin{center}
 \includegraphics[scale=1.25]{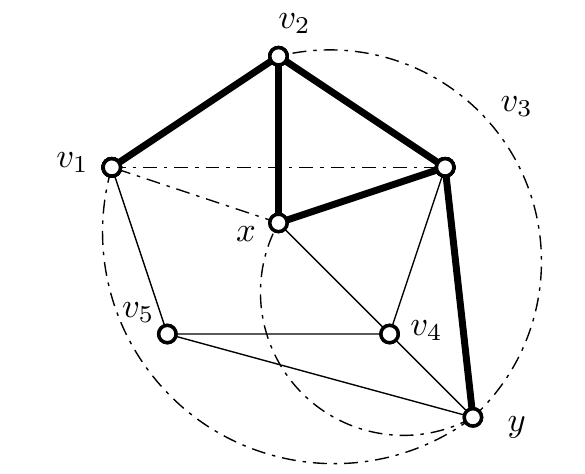}
 \caption{Subcase 2.2, Lemma \ref{lemma: comp. tri.free}. With $\{x,y,v_1\}\Seq I$ and $\{v_2,v_3,v_4\}\Seq N(x)$, one gets $v_2y\not\in E(G)$, since $G$ is claw-free. But then $G[\{v_1,v_2,v_3,x,y\}]$ will be a bull, a contradiction}\label{pic: c-b-free-case2-2}
\end{center}
\end{figure}
\begin{figure}[H]
\begin{center}
 \includegraphics[scale=1.25]{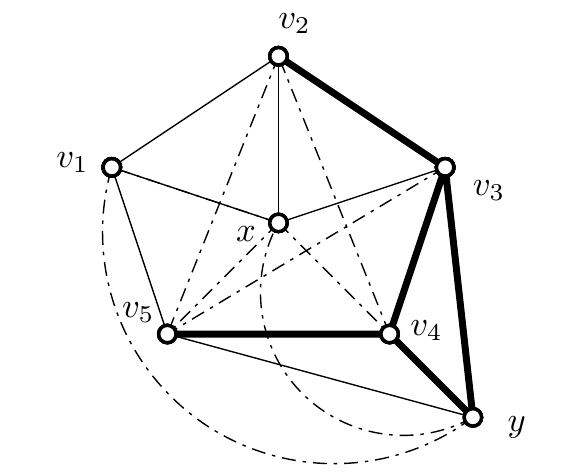}
 \caption{General situation in Subcase 2.3, Lemma \ref{lemma: comp. tri.free}. With $\{x,y,z\}\Seq I$, one may assume $xv_4\not\in E(G)$, $xv_5\not\in E(G)$, implying $xv_1,xv_2,xv_3\in E(G)$. One may further assume $yv_1\not\in E(G)$. As such, $yv_5\in E(G)$, for otherwise $G[\{v_3,v_4,v_5,x,y\}]$ would be a bull.}\label{pic: c-b-free-case2-3-general}
\end{center}
\end{figure}

\Subcase{2.3}{$\card{I\cap V(C)}=0$.}

Let $x,y,z$ be distinct vertices in $I$. Since $G$ is claw-free, no vertex of $C$ is adjacent to all three of $x,y,z$. Hence, by the pigeonhole principle and Lemmas \ref{neighboursincycles} and \ref{dominatingcycle}, we may assume $v_4\notin N(x)$ and $v_5\notin N(x)$, which imply $xv_1,xv_2,xv_3\in E(G)$. Furthermore, we may assume $v_1\notin N(y)$; thereby, $yv_3, yv_4\in E(G)$. But then we must have $yv_5\in E(G)$ for otherwise $G[\{v_3,v_4,v_5,x,y\}]$ would be a bull. If $v_1,v_5\in N(z)$ then $G[\{v_1,v_5,x,y,z\}]$ would be a bull, and if only one of $v_1,v_5$ is in $N(z)$ then $G[\{v_1,v_5,x,z\}]$ or $G[\{v_1,v_5,y,z\}]$ would be a claw. Hence, $v_1\notin N(z)$ and $v_5\notin N(z)$; thereby $v_2,v_3,v_4\in N(z)$. But then $G[\{v_3,v_4,v_5,x,z\}]$ would be a bull, a contradiction.
\begin{figure}[H]
\centering
\begin{minipage}[t]{0.50\textwidth}
\centering
\includegraphics[width=2.8in]{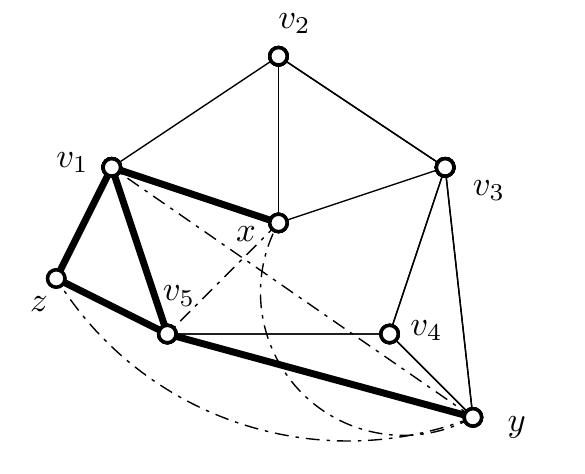}
\caption*{\textbf{(a)} If $v_1z,v_5z\in E(G)$, $G[\{v_1,v_5,x,y,z\}]$ would be a bull.}
\end{minipage}\hfill\begin{minipage}[t]{0.50\textwidth}
\centering
\includegraphics[width=2.8in]{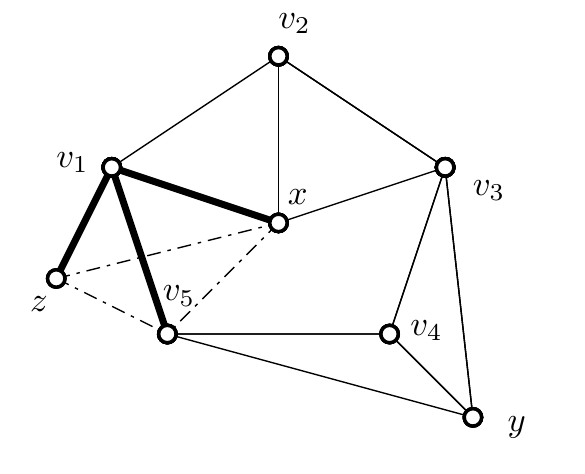}
\caption*{\textbf{(b)} If $v_1z\in E(G)$ and $v_5z\not\in E(G)$ then $G[\{v_1,v_5,x,z\}]$ would be a claw.}
\end{minipage}\hfill\begin{minipage}[t]{0.50\textwidth}
\centering
\includegraphics[width=2.8in]{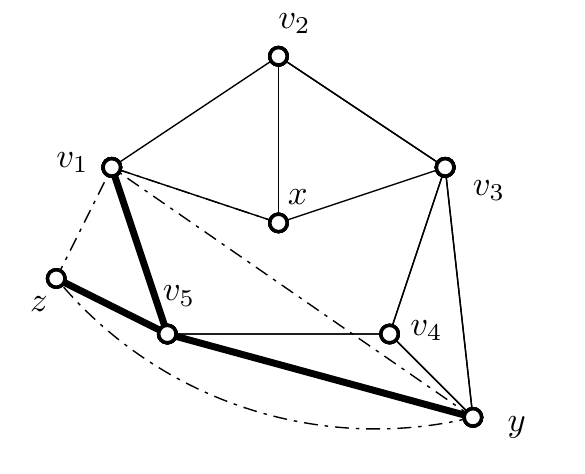}
\caption*{\textbf{(c)} If $v_5z\in E(G)$ and $v_1z\not\in E(G)$ then $G[\{v_1,v_5,y,z\}]$ would be a claw.}
\end{minipage}\hfill\begin{minipage}[t]{0.50\textwidth}
\centering
\includegraphics[width=2.8in]{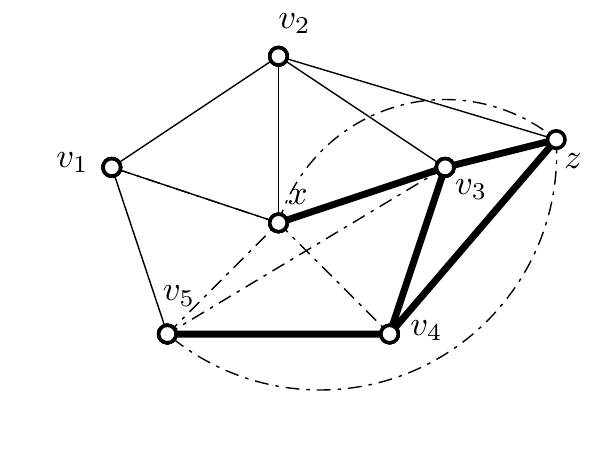}
\caption*{\textbf{(d)} Since $v_3z,v_4z\in E(G)$ and $v_5z\not\in E(G)$, $G[\{v_3,v_4,v_5,x,z\}]$ is a bull, a contradiction.}
\end{minipage}
\caption{Lemma \ref{lemma: comp. tri.free} with $\ell(G)=5$. From the general situation described in Figure \ref{pic: c-b-free-case2-3-general} one gets $\{v_3,v_4\}\Seq N(y)$ and $N(z)\cap V(C)=\{v_1,v_2,v_3\}$, leading to the bull in \textbf{(d)}.}\label{pic: Fig-c-b-free-case2-3}
\end{figure}
	\end{proof}
\section{The case $\ell(G)\le 3$ with $\alpha(G)\ge 3$.}\label{section: path-expansions}
\begin{lemma}\label{lemma: diam-chordal}
Let $G$ be a (claw, bull)-free graph with $\alpha(G)\ge 3$ and $diam(G)=2$. Then $\ell(G)\ge 6$.
\end{lemma}
\begin{proof}
Let $\{\alpha_1,\alpha_2,\alpha_3\}$ be an independent set of vertices in $G$. Since $diam(G)= 2$, for each $i\in[3]$ there is a common neighbor $w_i\in V(G)$ of the $\alpha_j$s for $j\in[3]\setminus\{i\}$. Moreover, for each $i\in[3]$ we have $w_i\alpha_i\not\in E(G)$, for otherwise $G[\{\alpha_1,\alpha_2,\alpha_3\}\cup\{w_i\}]$ would be a claw. We shall show that the 6-cycle $C:\;\alpha_1 w_3 \alpha_2 w_1 \alpha_3 w_2 \alpha_1$ is induced; thereby $\ell(G)\ge 6$. To this end, suppose on the contrary, that $C$ has a chord. As such, without loss of generality we may assume $w_2 w_3\in E(G)$. But then $G[\{\alpha_1,\alpha_2,\alpha_3,w_2,w_3\}]$ will be a bull, a contradiction. Hence, $C$ is an induced cycle, as desired.
\end{proof}
\begin{figure}[H]
\begin{center}
 \includegraphics[scale=1.1]{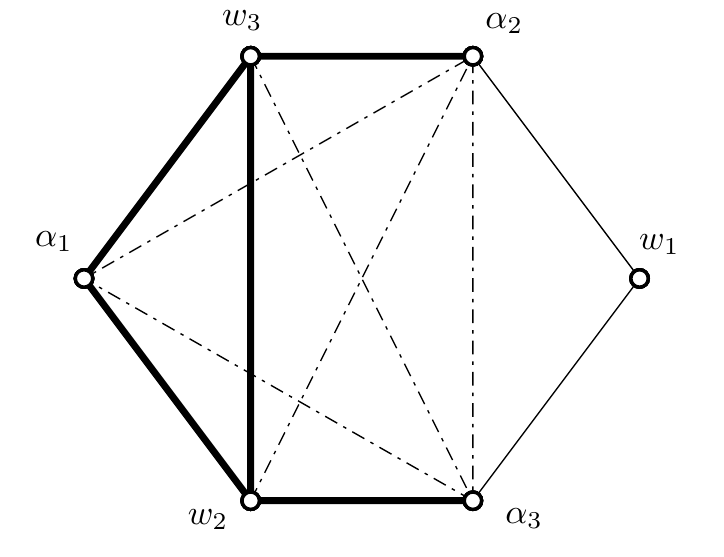}
 \caption{Cycle $C$ introduced in the proof of Lemma \ref{lemma: diam-chordal} is induced: If $w_2w_3\in E(G)$ then $G[\{\alpha_1,\alpha_2,\alpha_3,w_2,w_3\}]$ would be a bull, a contradiction.}\label{pic: c-b-free-diam-2-l-ge-6}
\end{center}
\end{figure}

The following proposition will be used in multiple occasions in the proof of Lemma \ref{shortcycles}.
\begin{prop}\label{cycle-chordal}
Let $u,v_0,v_1,\dots,v_k,u$ be a cycle in a graph $G$ with $\ell(G)\le 3$, such that $v_0,v_1,\dots,v_k$ is an induced path in $G$. Then $uv_i\in E(G)$ for each $i\in[k]$.
\end{prop}
The following Lemma is the main result of this subsection. 
\begin{lemma}\label{shortcycles}
\setcounter{claimCount}{0}
Let $G$ be a (claw, bull)-free graph with $\ell(G)\le 3$ and $\alpha(G)\ge 3$. Then:
\begin{enumerate}
\item $diam(G)\ge 4$; and
\item $G$ is an expansion of a path.
\end{enumerate} 
\end{lemma}
\begin{proof}
\setcounter{claimCount}{0}
Let $k=diam(G)$. According to lemma \ref{lemma: diam-chordal} we have $k\ge 3$. Let $v_0,v_k\in V(G)$ such that $d_G(v_0,v_k)=k$, and let $P=v_0,\dots,v_k$ be a geodesic path between them. Moreover, let $U=N_G(u_0)\setminus \{u\}$, set $H=G-U$ and define $N_i$'s as in Lemma \ref{lemma: structure cw-bl-free}. Moreover, let $A=\{\alpha_1,\alpha_2,\alpha_3\}$ be an independent set of vertices where $\alpha_i$s are distinct.

\begin{claim}\label{claim: u-to-Ni} No vertex in $U$ is adjacent to $v_3$ or a vertex in any $N_i$ with $3<i\le k$. Moreover, a vertex of $U$ adjacent to a vertex in some $N_i$ is adjacent to every vertex in every $N_j$ with $j <i$.
\end{claim}
\begin{claimproof}
If the first part does not hold, then one has $d_G(v_0,v_k)<2+k-3<k$, a contradiction. As for the second part of the claim, consider a vertex $u\in U$ which is adjacent to a vertex $w_i\in N_i$ and for each $j\in\{0,\dots,i-1\}$ choose a vertex $w_j\in N_j$. Then, by the definition of the $N_j$s, $w_0,w_1,\dots,w_i$ is an induced path. Since $uw_0=uv_0,uw_i\in E(G)$, by Proposition \ref{cycle-chordal} it follows that every $uw_j$ is an edge of $G$. This establishes the second part of the claim.
\end{claimproof}

\begin{claim}\label{claim: Nk+1empty} $N_i=\varnothing$ for $i>k$.
\end{claim}
\begin{claimproof}
It suffices to show that $N_{k+1}=\varnothing$. To this end, by the way of contradiction suppose $N_{k+1}\not=\varnothing$ and choose a vertex $w_{k+1}\in N_{k+1}$. 
Let $Q$ be a geodesic path in $G$ from $w_{k+1}$ to $v_0$. Considering the fact that $d_H(w_{k+1},v_0)=k+1>k$, we conclude that $Q$ must contain exactly one vertex, say $u$, from $U$. As such, we must also have $u v_0\in E(Q)$, i.e. $uv_0$ must be the last edge of $Q$. Moreover, since $V(Q)\setminus \{u\}\Seq V(H)$, every vertex in $V(Q)\setminus \{u\}$ must be in some $N_j$. Suppose the vertex of $Q$ preceding $u$ is in $N_i$ and call it $w_i$.
\medskip

\Case{I}{$i> k$.}

Set $w_3=v_3$ and for each $j\in(([i-1]\cup\{0\})\setminus\{3\})$ choose $w_j\in N_j$. Note that as $i\>k\ge 3$, the induced path $w_0,w_1,\dots,w_i$ contains $v_3$. Moreover, since $uw_0=uv_0\in E(G)$ and $uw_i\in E(G)$, we must have $uw_j\in E(G)$ for each $j\in\{0,\dots,i\}$; in particular, $uv_3\in E(G)$. But the latter contradicts Claim \ref{claim: u-to-Ni}. Hence, this case does not happen.
\medskip

\Case{II}{$i\le k$.}

$Q$ will be of the form $ w_{k+1}w_k,\dots,w_i,u,v_0$ where each $w_j$ ($j\in\{i,i+1,\dots,k+1\}$) is in $N_j$. In particular the length of $Q$, which is bounded above by the diameter $k$ of $G$, is $k+3-i$. Hence, $i\ge 3$. On the other hand, by Claim \ref{claim: u-to-Ni}, we must have $i<4$ (since $u$ is not adjacent to $v_3\in N_3$). Therefore, $i=3$ and, hence, $uv_1,uv_2\in E(G)$, according to Claim \ref{claim: u-to-Ni}. Moreover, we have $uw_4\not\in E(G)$, by Claim \ref{claim: u-to-Ni}, whereas $v_2w_3,w_3w_4\in E(G)$ and $v_0v_2,v_0w_3,v_0w_4,v_2w_4\not\in E(G)$. Thus, $G[\{v_0,v_2,w_3,w_4,u\}]$ will be a bull, a contradiction.

\end{claimproof}
\begin{claim}\label{Ni+U-is-all} $V(G)=(\bigcup_1^k N_i)\cup U$.
\end{claim}
\begin{claimproof} Contrary to the claim, assume $(\bigcup_1^k N_i)\cup U\subsetneq V(G)$ or, equivalently, $W:=N(U)\setminus(\bigcup_1^k N_j)\not=\varnothing$. Let $\mathscr{R}$ be the set of paths of the shortest length from a vertex in $W$ to $v_k$. Note that every path in $\mathscr{R}$ has at least one vertex in common with $U$, for otherwise $w$ would be in $\bigcup_1^k N_j$, a contradiction. Choose $R\in\mathscr{R}$ such that $\card{V(R)\cap U}$ is the minimum. Furthermore, let $w$ be the initial vertex of $R$ and $u$ the last vertex of $R$ which is in $U$. Observe that every vertex of $R$ that follows $u$ is in some $N_j$ with $j\in[k]$ and, according to Claim \ref{claim: u-to-Ni}, the immediate successor of $R$ is in $\bigcup_1^3 N_j$. Let the latter be $w_i\in Ni$. Then, we must have
\begin{equation}
R(u,v_k)=\left\{\begin{array}{lcl}
u,w_{i},\dots,w_{k-1},v_k & \mbox{if $i<k-1$;}\\
u,w_{k-1},v_k & \mbox{if $i=k-1$;}\\
u,w_3,v_3 & \mbox{if $i=k=3$};
\end{array}\right.
\end{equation}
where each $w_j$ is in $N_j$ and $w_3\not=v_3$. (Recall that $uv_3\not\in E(G)$; thereby, in case $i=3$ we must have $w_3\not= v_3$.) As a result, the length of $R(u,v_k)$ is the $\max\{k-i+1,2\}$. Then. from the facts that
\begin{itemize}
\item $R$ has at least one edge more than $R(u,v_k)$,
\item length of $R$ is bounded above by the diameter $k$, and
\item $i\le 3$,
\end{itemize}
it follows that $i\in \{2,3\}$. In particular, $v_0 w_i\not\in E(G)$.

Consequently, if $i=2$ or $i=k=3$ then we must have $wu\in E(G)$ (for otherwise the length of $R$ would be grater than k); hence, $G[\{w,u,v_0,v_i\}]$ would be a claw, a contradiction. Also, $G[\{w,u,v_0,v_i\}]$ would be a claw if $i=3$, $k>3$ and $wu\in E(G)$. Hence, the only case to examine is when $i=3<k$ and $wu\not\in E(G)$. As such, that $R$ has length $\le k$ implies  
\begin{equation}
R=
\left\{\begin{array}{lcl}
w,u',u,w_{3},v_4 & \mbox{if $k=4$;}\\
w,u',u,w_{3},\dots,w_{k-1},v_k& \mbox{if $k>4$};\
\end{array}\right.
\end{equation}
\begin{figure}[h]
\begin{center}
 \includegraphics[scale=1.35]{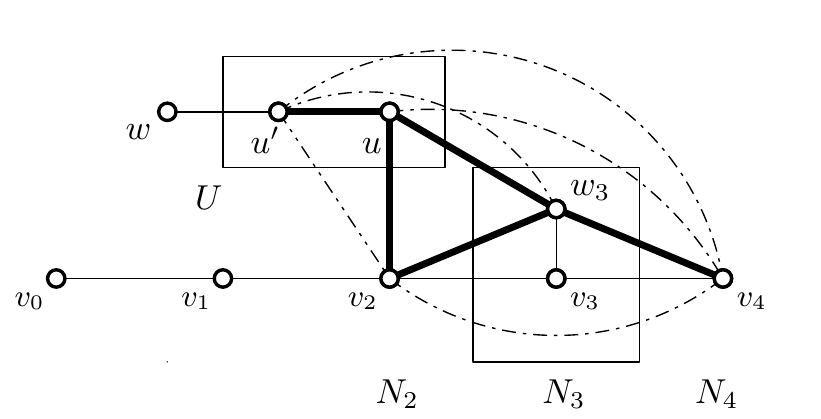}
 \caption{Ruling out the case $i=3<k$ and $wu\not\in E(G)$ in the proof of Claim \ref{Ni+U-is-all}, Lemma \ref{shortcycles}. With $R(w,w_i)=w,u',u,w_3$, $G[\{u'u,v_2,v_4,w_3\}]$ will be a bull.} \label{pic: c-b-free-lemma-last}
\end{center}
\end{figure}
Note that we must have $u'\in U$, for otherwise $R(u',v_k)$, would be in $\mathscr{R}$, contradicting the choice of $R$ as a path of the shortest length in $\mathscr{R}$. Likewise, we must have $u'w_3\not\in E(G)$, for otherwise $wu'+u'w_3+R(w_3,v_k)$ would be a path in $\mathscr{R}$ yet shorter than $R$. Furthermore, we must have $u'v_2\not\in E(G)$, for otherwise the path $R':=wu'+u'v_2+v_2w_3+R(w_3,v_k)$ would have the same length as $R$, implying $R'\in\mathscr{R}$, with the property that 
\begin{equation*}
\card{V(R')\cap U}<\card{V(R)\cap U},
\end{equation*}
contradicting the choice of $R$ as an element in $\mathscr{R}$ with minimum size intersection with $U$. But then $G[\{u'u,v_2,v_4,w_3\}]$ will be a bull, a contradiction.
\end{claimproof}



\begin{claim}\label{claim: u-to-N2-clique} Let $U'\subseteq U$ such that any two vertices in $U'$ have a common neighbor in $N_2$. Then $U'$ is a clique.
\end{claim}
\begin{claimproof}
According to Claim \ref{claim: u-to-Ni} no vertex in $U'$ is adjacent to $v_3$. Hence, for any pair $x,y$ of distinct vertices in $U'$ with $xy\not\in E(G)$, and for every common neighbor $w_2\in N_2$ of $x,y$ the graph $G[\{x,y,w_2,v_3\}]$ is a claw. Therefore, $U'$ must be a clique.
\end{claimproof}

\begin{claim}\label{claim: diam=3}
If there is a vertex $u\in U$ such that $N_G(u)\cap N_3\not=\varnothing$ then $diam(G)=3$.
\end{claim}
\begin{claimproof}
Consider any vertex $w_3\in N_G(u)\cap N_3$. According to Claim \ref{claim: u-to-Ni} we have $w_3\not=v_3$, and $uv_2\in E(G)$. If, in addition, $diam(G)\ge 4$ or, equivalently, if $N_4\not=\varnothing$, then $G[\{u,v_0,v_2,w_3,v_4\}]$ would a bull, a contradiction. Hence, Hence, we must have $diam(G)=3$.
\end{claimproof}

\begin{claim}\label{claim: u-tov1} $U\Leftrightarrow \{v_1\}$.
\end{claim}
\begin{claimproof}
Let $u\in U$ such that $uv_1\not\in E(G)$. Then, by Claim \ref{claim: u-to-Ni} $u$ is adjacent to no vertex in an $N_i$ with $i>0$; in other words, we have
\begin{equation}\label{eq: claim 6-01}
N_G(u)\Seq U\cup \{v_0\}.
\end{equation}
Let $\mathscr{Q}$ be the set of paths of the shortest length from $u$ to $v_k$. Note that every path $Q\in\mathscr{Q}$ has at least two vertices in $U$ (one of which is of course $u$), for otherwise one would have $uv_0\in Q$, implying that $l(Q)=l(Q(v_0,v_k))+1>k$, a contradiction. Hence,
\begin{equation}\label{eq: claim 6-02}
\card{V(Q)\cap U}\ge2\qquad\forall Q\in\mathscr{Q}.
\end{equation}
Choose $Q'\in\mathscr{Q}$ such that $\card{V(Q')\cap U}$ is the minimum and let $u'$ be the last vertex of $Q'$ which is in $U$. Note that
\begin{equation}\label{eq: claim 6-03}
l(Q'(u',v_k))\ge 
\left\{\begin{array}{lcl}
k-2 & \mbox{if $k>3$;}\\
k-1& \mbox{if $k=3$};\
\end{array}\right.
\end{equation}
where the second inequality follows from the fact that $u'v_3\not\in E(G)$. Note that $u'$ must be adjacent to some vertex $w_2\in N_2$, for otherwise one would have $l(Q')>k+1$, a contradiction. As such, we must have
\begin{equation}\label{eq: claim 6-04}
uu'\not\in E(G),
\end{equation}
for otherwise $G[\{u,u',v_1,v_3,w_2\}]$ would be a bull. (See Figure \ref{pic: c-b-free-lemma-last02}.) 
\begin{figure}[H]
\begin{center}
 \includegraphics[scale=1.35]{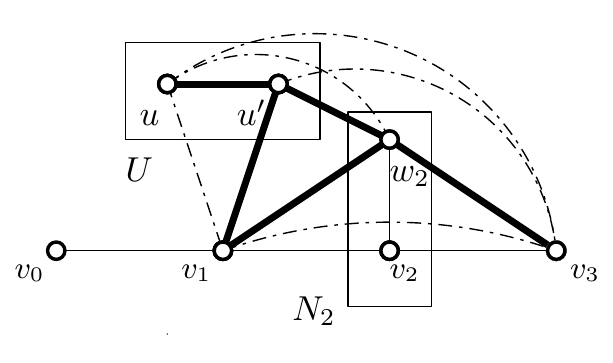}
 \caption{Ruling out the case that $\card{V(Q')\cap U}= 1$ in the proof of Claim \ref{claim: u-tov1}, Lemma \ref{shortcycles}. For every $w_2\in N_2 \cap N_G(u'')$, the graph $G[\{u,u',v_1,v_3,w_2\}]$ will be a bull.} \label{pic: c-b-free-lemma-last02}
\end{center}
\end{figure}

\noindent Thus, according to \ref{eq: claim 6-03} we must have $k>3$. Moreover, 
\begin{equation}\label{eq: claim 6-05}
\exists\;u''\in U:\quad Q'(u,u')=\;u,u'',u',
\end{equation}
and $u'$ is followed by a vertex $w_3\in N_3$ along $Q'$. Note that
\begin{equation}\label{eq: claim 6-05+1}
 u''w_3\not\in E(G),
\end{equation} for otherwise the path from $u$ to $v_k$ obtained by augmenting the path $u,u',w_3$ to $Q'(w_3,v_k)$ would be shorter than $Q'$, a contradiction. Moreover, as such, we must have
\begin{equation}\label{eq: claim 6-06}
u''v_2\not\in E(G),
\end{equation}
\begin{figure}[H]
\begin{center}
 \includegraphics[scale=1.35]{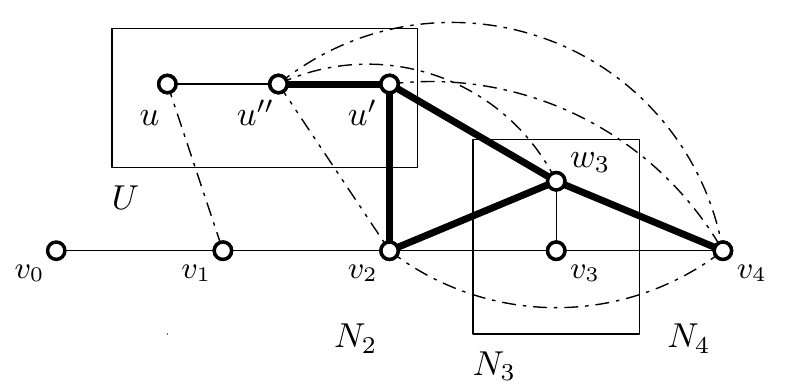}
 \caption{Ruling out the case that $\card{V(Q')\cap U}\ge 2$ in the proof of Claim \ref{claim: u-tov1}, Lemma \ref{shortcycles}. By \eqref{eq: claim 6-05+1} and \eqref{eq: claim 6-06} the graph $G[\{u',u'',v_2,v_4,w_3\}]$ will be a bull.} \label{pic: c-b-free-lemma-last03}
\end{center}
\end{figure}
\noindent for otherwise the path $Q''$ obtained by augmenting the path $u,u'',v_2,w_3$ to $Q(w_3,v_k)$ will have the same length as $Q'$ whereas
\begin{equation*}
\card{V(Q'')\cap U}< \card{V(Q')\cap U},
\end{equation*}
contradicting the choice of $Q'$. Finally, as shown in Figure \ref{pic: c-b-free-lemma-last03}, $G[\{u',u'',v_2,v_4,w_3\}]$ will be a bull, a contradiction. Hence, $U\Leftrightarrow \{v_1\}$, as desired.
\end{claimproof}
\begin{claim}\label{claim: U-clique} $U$ is a clique.
\end{claim}
\begin{claimproof}
Suppose, contrary to the claim, that $x,y$ are distinct vertices in $U$ such that $xy\not\in E(G)$. By Claim \ref{claim: u-tov1} we have
\begin{equation}
xv_1,yv_1\in E(G)
\end{equation}
Moreover, we have $xv_2\in E(G)$ or $yv_2\in E(G)$, for otherwise $G[\{x,yv_1,v_2\}]$ would be a claw. In addition, according to Claim \ref{claim: u-to-N2-clique}, $v_2$ cannot be adjacent to both $x$ and $y$.  Hence, we may assume
\begin{equation}
xv_2\not\in E(G)\quad \& \quad yv_2\in E(G).
\end{equation}
But then, $G[\{x,y,v_1,v_2,v_3\}]$ would be a bull, a contradiction. Hence, $U$ is a clique. 
\end{claimproof}

\noindent\textbf{(a)} Let $I$ be a largest independent set in $G$. By Claim \ref{claim: U-clique} we have $\card{I\cap U}\le 1$. Hence, by Lemma \ref{lemma: structure cw-bl-free}, Claim \ref{Ni+U-is-all} and that $\card{I}\ge 3$, we must have $k\ge 4$, as desired.\\
\textbf{(b)} As $k\ge4$ and according to Claims \ref{claim: u-to-Ni} and \ref{claim: diam=3}, no vertex in $U$ is adjacent to a vertex in any $N_i$ with $i\ge 3$. Note that by Claim \ref{claim: u-tov1}, we have $uv_1\in E(G)$ for every $u\in U$. We shall show that every vertex in $U$ is either adjacent to every vertex in $N_2$ or non-adjacent to every vertex in $N_2$. To this end, by the way of contradiction, let there be $u\in U$ and $s_2,t_2\in N_2$ such that $us_2\in E(G)$ and $ut_2\not\in E(G)$. Then $G[\{s_2,t_2,u,v_3,v_4\}]$ will be a bull, a contradiction. Therefore, $U$ is the disjoint union of the sets $V_0:=\{u\in U: \{u\}\Leftrightarrow N_2\}$ and $V_1:=\{u\in U: \nexists w\in N_2:\; uw\in E(G)\}$, and $G$ is the expansion of the path $v_0,\dots,v_k$ where each vertex $v_i$ is replaced by the bag $M_i$ defined by $M_i=N_i\cup V_i$ for $i=0,1$ and $M_i=N_i$ for $i>1$. 
\end{proof}
\begin{figure}[H]
\begin{center}
 \includegraphics[scale=1.35]{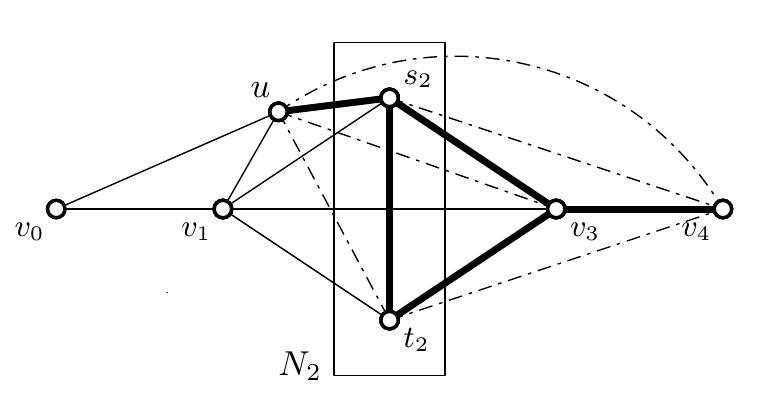}
 \caption{proof of part \textbf{(b)} of Lemma \ref{shortcycles}; showing that $V_0\cap V_1=\varnothing$. If $u\in U$ ,$s_2,t_2\in N_2$ with $us_2\in E(G)$ and $ut_2\not\in E(G)$, then $G[\{s_2,t_2,u,v_3,v_4\}]$ will be a bull} \label{pic: c-b-free-lemma-last04}
\end{center}
\end{figure}

\section{Proof of Theorem \ref{thm:main}}
	
		

\begin{proof}[Proof of Theorem \ref{thm:main}]
It is easy to check that an expansion of a path, that of a cycle, and the complement of a triangle-free graph are all (claw, bull)-free. Conversely, by Lemmas \ref{longcycles}, \ref{lemma: comp. tri.free}, and \ref{shortcycles}, every (claw, bull)-free graph  is either an expansion of a cycle of length $\ge 6$, or the complement of a triangle-free graph, or an expansion of a path of length $\ge 4$.
\end{proof}

\bibliographystyle{plain}






\end{document}